\def\dsp{\displaystyle}
\newenvironment{proof}{{\noindent \textbf{Proof}\,\,}}{\hspace*{\fill}$\Box$\medskip}

\newtheorem{Def}{Definition}
\newtheorem{Thm}{Theorem}

\newtheorem{Lem}{Lemma}

\newtheorem{Rem}{Remark}
\newtheorem{Prop}{Proposition}

\documentclass[12pt] {article}
\usepackage{amsfonts,graphicx,amsmath,amscd, amssymb}

\def\zz{\mathbb{Z}}
\def\cc{\mathbb{C}}
\def\rr{\mathbb{R}}

\def\F{\mathcal F}

\def\G{\mathcal G}

\def\A{\mathcal A}

\def\e {\varepsilon }

\def\ph {\varphi }

\def\diffeo {diffeomorphism }

\def\homeo {homeomorphism }
\def\diffeos {diffeomorphisms }

\def\skpr {skew product }
\def\skprs {skew products }
\def\st {such that }
\def\nbd {neighborhood }

\def\om {\omega }
\def\a {\alpha }
\def\skpr {skew product }

\def\l {\lambda }
\def\d {\delta }

\def\L {\Lambda }

\def\om {\omega}

\def\s {\sigma }
\def\Ast {A_{stat}}

\def\a{\alpha}

\def\p {\partial }

\def\ga {\gamma}
\def\t{\tilde}

\def\zz{\mathbb{Z}}
\def\cc{\mathbb{C}}
\def\rr{\mathbb{R}}
\def\F{\mathcal F}
\def\G{\mathcal G}

\def\L{\mathcal L}

\def\e {\varepsilon }

\def\ph {\varphi }

\def\a {\alpha}

\def\t {\tilde }
\def\ga {\gamma }

\def\diffeo {diffeomorphism }
\def\homeo {homeomorphism }

\def\diffeos {diffeomorphisms }

\def\skpr {skew product }
\def\skprs {skew products }
\def\st {such that}
\def\nbd {neighborhood }

\def\om {\omega }
\def\skpr {skew product }

\def\l {\lambda }
\def\L {\Lambda }

\def\st{such that }

\def\mlip{\mbox { Lip }}

\def\nn{\mathbb N}

\def\Am {A_{max}}

\def\ns{North-South skew product }
\def\nss{North-South skew products }

\title{Invisible parts of attractors}

\author{Yu. Ilyashenko\thanks{ The author was supported by part by the grants
NSF 0700973, RFBR 07-01-00017-à, RFFI-CNRS 050102801} \thanks
{Cornell University, US; Moscow State and Independent Universities,
Steklov Math. Institute, Moscow} and A. Negut \thanks{Princeton
University, US; Independent University of Moscow, Russia }}

\date{}

\begin{document}

\maketitle

\begin{abstract}
This paper deals with the attractors of generic dynamical systems. We introduce the notion of $\e-$invisible set, which is an open set in which almost all orbits spend on average a fraction of time no greater than $\e$. For extraordinarily small values of $\e$ (say, smaller than $2^{-100}$), these are areas of the phase space which an observer virtually never sees when following a generic orbit.

We construct an open set in the space of all dynamical systems which have an $\e-$invisible set that includes parts of attractors of size comparable to the entire attractor of the system, for extraordinarily small values of $\e$. The open set consists of $C^1$ perturbations of a particular skew product over the Smale-Williams solenoid. Thus for all such perturbations, a sizable portion of the attractor is almost never visited by generic orbits and practically never seen by the observer.
\end{abstract}

\section{Introduction}

One of the major problems in the theory of dynamical systems is the
study of the limit behavior of orbits. Most orbits tend to
invariant sets called attractors. Knowledge of the attractors may
indicate the long time behavior of the orbits.

Yet it appears that some large parts of attractors may be
practically invisible. In this paper we describe an open set in the
space of dynamical systems whose attractors have a large
unobservable part. Precise definitions follow.

\subsection{Attractors and $\e$-invisible open sets}

There are different nonequivalent definitions of attractors.

Let  $X$ be a \emph{metric measure space}, with a finite measure
$\mu $. This measure will not necessarily be probabilistic, but we
will assume that $\mu (x) \ge 1$.

Often, but not always, $X$ will be  a compact smooth manifold with
or without boundary. In this case the metric is the geodesic
distance and the measure is the Riemannian volume. In the following,
$d$ and $\mu$ will denote the distance and the measure on $X$. The
following definitions all concern maps $F:X\rightarrow X$ which are
homeomorphic onto their image.

\begin{Def} [Maximal attractor]  \label{def:amax}
An invariant set $A_{max}$ of $F$ is called \emph{a maximal
attractor in its neighborhood} provided that there exists a \nbd $U$
of $\Am $ \st
$$
A_{max} = \bigcap_{n=0}^\infty F^n(U).
$$
\end{Def}

\begin{Def} [Milnor attractor, \cite{M85}] The \emph{Milnor attractor} $A_M$ of $F$ is the minimal invariant
closed set that contains the $\om$-limit sets of almost all points.
\end{Def}

\begin{Def} [Statistical attractor, \cite{AAIS85}]   \label{def:astat}
The \emph{statistical attractor} $A_{stat}$ of $F$ is the minimal
closed set \st almost all orbits spend an average time of $1$ in any
\nbd of $\Ast$.
\end{Def}

\begin{Def} [$\e-$invisible open set] \label{def:e-invis}
An open set $V \subset X$ is called $\e-$invisible if the orbits
of almost all points visit $V$ with average frequency no grater than
$\e$:

\begin{equation}\label{eqn:strong}
\limsup_{N\to \infty } \frac {|\{0\leq k< N | F^k(x)\in
V\}|}N\leq \e.
\end{equation}
\end{Def}

\subsection{Skew products}

In this section, $X$ is a Cartesian product $X = B \times M$
with the natural projection $ \pi : X \to M $ along $B$. The set $B$
is the base, while $M$ is the fiber. Both $B$ and $M$ are metric measure spaces. The distance between two points of $X$ is,
by definition, the sum of the distances between their projections onto
the base and onto the fiber. The measure on $X$ is the
Cartesian product of the measures of the base and of the fiber.

Maps of the form
\begin{equation}   \label{eqn:smooth1}
F : ( b ,x) \mapsto (h(b),f_b(x))
\end{equation}
are called \emph{skew products} on $X$. Denote by $C^1_p$ ($p$ stands for
product) the space of all skew products on $X$, with distance given
by

\begin{equation}
\label{eqn:distskew} d_{C^1_p}(F ,\t F ) = \max_B d_{C^1}(f_b^{\pm
1}, \t f_b^{\pm 1}).
\end{equation}
\begin{Def} \label{def:moder} A \homeo $F$ of a metric space is
called $L-$moderate if $\mlip F^{\pm 1} \le L$ (here $\mlip$ denotes
Lipschitz constant).
\end{Def}

We shall consider only $L-$moderate maps $F$ with $L \le 100$, in
order to guarantee that the phenomenon of $\e-$invisibility is not
produced by any extraordinary distortion in the maps $F$ or
$F^{-1}$.

\subsection{Skew products over the Smale-Williams solenoid and the main result} \label{sub:solenoid}

Take $R\geq 2$, and let $B=B(R)$ denote the solid torus
$$
B = S^1_{\bold y} \times D(R), \ S^1_{\bold y} = \{ y \in \rr /\zz
\}, \ D(R) = \{z \in \mathbb C| |z| \le R \}.
$$
The \emph{solenoid map} is defined  as
\begin{equation}
h=h_{\lambda}: B \to B, (y, z) \mapsto (2y, e^{2\pi i y} + \lambda
z), \lambda < 0.1. \label{eqn:sol}
\end{equation}
The exact values of the parameters $R$ and $\lambda$ are not
crucial, since the dynamics of the map $h$ is the same regardless of
their particular values.

Let us consider the Cartesian product $X=B\times S^1$, where
$S^1=\rr/2\zz$. All skew products in this section are over this
Cartesian product, and the map $h$ in the base $B$ will always be the
solenoid map. Fix some $L\leq 100$, and let $D_L(X)$ denote the
space of $L-$moderate smooth maps $\G:X\rightarrow X$. Our main result on attractors is the
following Theorem.

\begin{Thm} [Main Theorem 1] \label{thm:mt} Consider any $n \ge 100$, and let $\nu=\frac 1n$. Then there exists
a ball $Q_n$ in the space $D_L(X)$ with the following property. Any
map $\G \in Q_n$ is structurally stable and has a statistical
attractor $A_{stat}=A_{stat}(\G)$ such that the following hold:

\begin{enumerate}

\item the projection $\pi(A_{stat})\subset S^1$ is a circular arc such that
\begin{equation}    \label{eqn:proj1}
[\nu, 1 - \nu] \subset \pi(A_{stat}) \subset [-\nu, 1 + \nu];
\end{equation}
\item the set $\pi^{-1}(0, \frac 1 4)$ is $\e-$invisible for $\G$ with
\begin{equation}    \label{eqn:invis}
 \e = 2^{-n}.
\end{equation}
\end{enumerate}
\end{Thm}

\begin{Rem} We do not make any quantitative statements about the size
of the ball $Q_n$ in $D_L(X)$. However, if we restrict attention to the
smaller space $C^1_{p,L}$ of $L-$moderate skew products over the solenoid,
then the conclusion of the above Theorem holds for a ball of radius
$\frac 1{4n^2}$ inside $C^1_{p,L}$. This is proven in an analogous
fashion to Theorem~\ref{thm:mtskew} below. Moreover, one can combine
Theorem~\ref{thm:mt} and Theorem~\ref{thm:mtskew} in order to obtain
the conclusion of Theorem~\ref{thm:mt} with $Q_n$ replaced by the
neighborhood in $D_L(X)$ of a ball of radius $\frac 1{4n^2}$ in
 $C^1_{p,L}$. For the sake of conciseness, we will avoid the
  technical details that produce this stronger result.
\end{Rem}

\begin{Rem} It is easy to construct a map with a large $\e-$invisible
part of its attractor and with distortion of order
$\e^{-1}$ (so with an enormous Lipschitz constant). Indeed, consider
an irrational rotation $R$ of a circle. The statistical attractor of
$R$ is the whole circle. Take a small arc of length $\e $ and a
coordinate change $H: S^1 \to S^1$ that expands this arc to a
semicircle $U$. Suppose that on the other half of the circle, the
Lipshitz constant of the inverse map $H^{-1}$ is no greater that
$3$. Then all the orbits of the map $f = H\circ R \circ H^{-1}$
visit the semicircle $U$ with frequency $\e $. Hence, this large
part of attractor is $\e-$invisible. However, the map $f$ has a
Lipshitz constant of order $\e^{-1}$. We reject such examples,
because they rely on extraordinarily large distorsions  to produce
$\e$-invisible sets, for extraordinarily small $\e$.

On the contrary, in Theorem~\ref{thm:mt} we construct maps on a
``human'' scale that produce $\e-$invisible sets, for
extraordinarily small $\e$. Indeed, our main theorem claims the
existence of  large $\e-$invisible sets with $\e $ arbitrarily small,
when the Lipschitz constant of the maps in question is uniformly
bounded (say, by $L = 100$).
\end{Rem}

\begin{Rem} Consider the bifurcation of a homoclinic orbit of a
saddle-node singular point in a family of planar vector fields $v_\e
$. For the critical parameter value $\e = 0$ the field has a
polycycle $\ga $ formed by the point and its homoclinic orbit, see
Fig~\ref{fig:bifsad}.

\begin{figure}[ht]
 \begin{center}
  \includegraphics[scale=0.5]{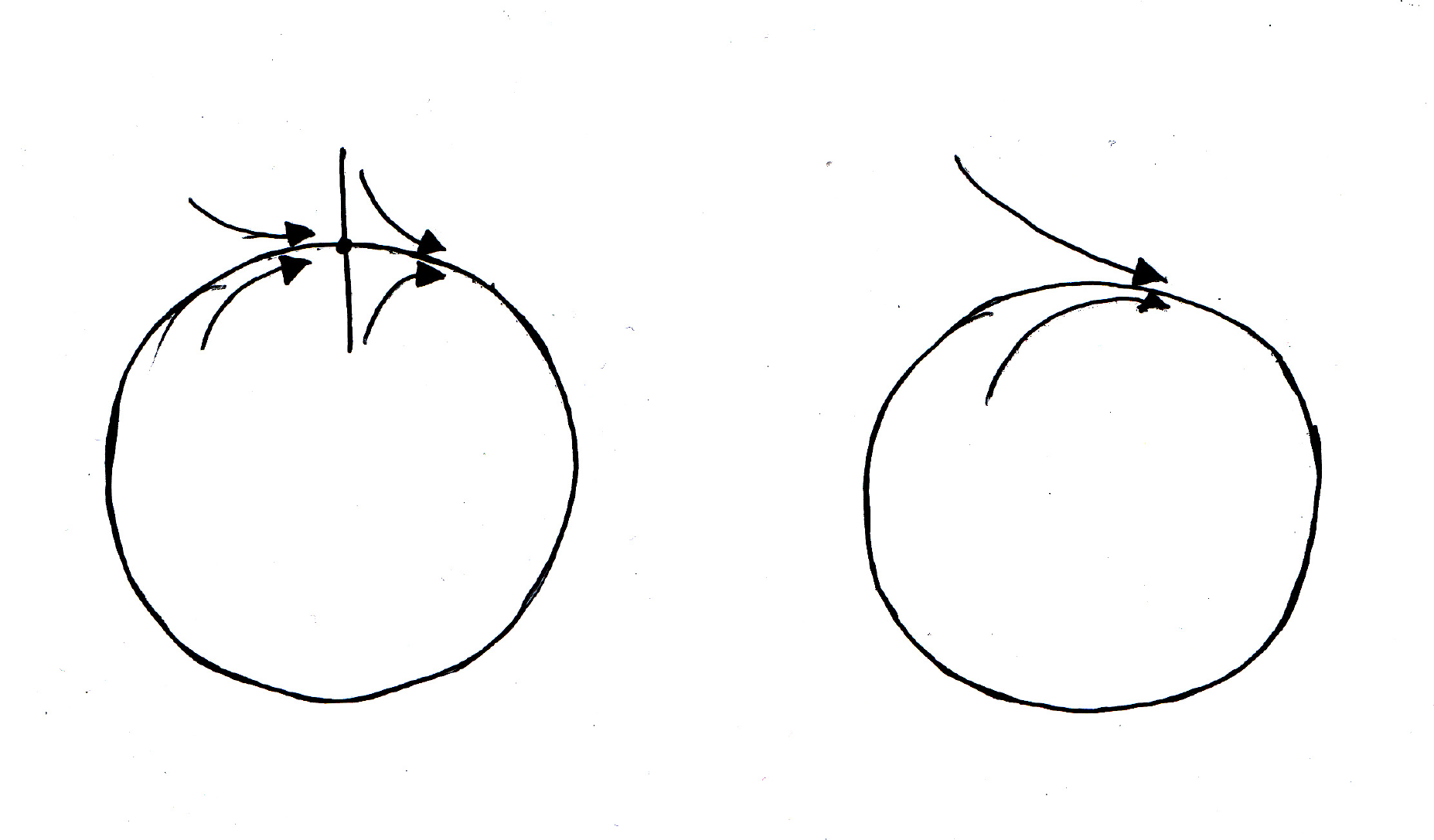}
   \caption{Bifurcation of a saddle-node orbit}
   \label{fig:bifsad}
  \end{center}
\end{figure}
  
For the postcritical parameter values, the field has an attracting
periodic orbit close to $\ga $. The finitely smooth normal form of
the family $v_\e $ near the saddle-node is:
$$
\dot x =\frac {x^2+\e }{1+ax}, \ \dot y = -y,
$$
for some $a \in \rr $, see  \cite{IYa91}. For any $\delta $, the
time spent by the orbits with initial condition on $\{ 0\} \times
[-\frac {\delta }{2}, \frac {\delta }{2}]$ in the \nbd ${[-\delta
,\delta ]}^2$ of zero is of order $\e^{\frac 1 2}$. Consider the
time one phase flow transformation $F_\e $ of the field $v_\e $. For
maps $F_\e $ with parameter values $\e \in (0, \e^2_0]$, the
shadowed area in Fig~\ref{fig:bifsad} is $\e_1-$invisible with
$\e_1$ close to $\e_0$. So, if $\e_1$ is extraordinarily small and
the parameter $\e $ is of order $\e^2_1$, we meet again the effect
of $\e-$invisibility of a large part of the attractor of the map
$F_\e $. But the parameter in the data is extraordinarily small
itself.
\end{Rem}

In our Main Theorem, the maps with large $\e-$invisible parts of
their attractor have a moderate Lipschitz constant $L \le 100$. We believe that the open set of such maps contains a ball of
radius $n^{-b}$ (for some universal constant $b \le 4$) in the whole
space $D_L(X)$. To prove this, one should replace the qualitative
arguments of Section~\ref{sec:pert} by quantitative ones.

Theorem~\ref{thm:mt} considers maps defined by ``human-scale
parameters of order $n$''. Formally speaking, this means that these
maps are at distance at least $n^{-a}$ from structurally unstable
maps: in $C^1_p$ this is proved for $a = 2$; in $D_L(X)$ we
conjecture it for $a = 3$. On the other hand, these maps have large
$\e-$invisible parts of attractors for $\e  = 2^{-n}$. In
Theorem 1, this part is equal to $A_{stat} \cap V$, and its size is
comparable with the size of the whole attractor. More precisely, the
projection $\pi (A_{stat} \cap V)$ is an arc of about $\frac 14 $ of the total length
of the arc $\pi (A_{stat})$. Roughly speaking, to visualize this part of the
attractor, the observer would have to pursue orbits for time intervals of
order $2^n$. Even for $n = 100$, it is hard to imagine such an
experiment.

\section{Invisible parts of attractors for skew products}

Skew products may be called \emph{miniUniverses of Dynamical
Systems}. Many properties observed for these products appear to
persist as properties of diffeomorphisms for open sets in various
spaces of dynamical systems. This heuristic principle was justified
in \cite{GI99}, \cite{GI00}, \cite{G06}.  In this context an open
set of \diffeos with nonhyperbolic invariant measures was found in
\cite{GIKN05} and \cite{KN07}, while other new robust properties of
\diffeos were described in \cite{GI99} and \cite{GI00}. The present
paper is another application of this heuristic principle.

In this section we define and study \skprs over the Bernoulli shift,
which closely mimic the dynamics of skew products over the
solenoid.

\subsection{Step and mild skew products over the Bernoulli shift}
 \label{sub:stepmild}

Let $\Sigma^2$ be the space of all bi-infinite sequences of $0$ and
$1$, endowed with the standard metric $d$ and $(\frac 1 2, \frac 1
2)-$probability Bernoulli measure $P$. In other words, if we take
$\om,\om'\in \Sigma^2$ given by
$$
\om = \dots \om_{-n} \dots \om_0 \dots \om_n \dots
$$
$$
\om' = \dots \om'_{-n} \dots \om'_0 \dots \om'_n \dots ,
$$
then
\begin{equation}\label{eqn:dist} d(\omega , \omega' ) = 2^{-n}
\mbox { where } n = \min \{ |k|, \textrm{ such that }\omega_k \ne
\omega'_k\},
\end{equation}
\begin{equation}\label{eqn:meas}
P(\{\omega, \textrm{ such that } \omega_{i_1}=\alpha_1, ...,
\omega_{i_k}=\alpha_k\}) = \dsp \frac 1{2^k},
\end{equation}
for any $i_1,...,i_k \in \zz$ and any $\alpha_1,...,\alpha_k\in
\{0,1\}$.

Let $\sigma :\Sigma^2 \to \Sigma^2 $  be the Bernoulli shift
$$
\sigma : \omega \mapsto \omega' , \ \omega'_n = \omega_{n+1}.
$$
A \emph{skew product over the Bernoulli shift} is a map
\begin{equation}   \label{eqn:mild}
G: \Sigma^2 \times M \to \Sigma^2 \times M, \ (\om ,x)\mapsto (\s
\om, g_{\om}(x)),
\end{equation}
where the \textit{fiber maps} $g_\omega$ are \diffeos of the fiber
onto itself. An important  class of skew products over the Bernoulli
shift consists of the so called \emph{step skew products}. Given two
diffeomorphisms $g_0,g_1:S^1\rightarrow S^1$, the step skew product
over these two diffeomorphisms is
\begin{equation}   \label{eqn:step}
G: \Sigma^2 \times M \to \Sigma^2 \times M, \ (\om ,x)\mapsto (\s
\om, g_{\om_0}(x)).
\end{equation}
Thus the fiber maps $g_\om$ only depend on the digit $\om_0$, and
not on the whole sequence $\om$. In contrast to step skew products,
general skew products where the fiber maps depend on the whole
sequence $\omega$ will be called \emph{mild} ones.


\subsection{SRB measures and minimal attractors}\label{sub:srb}

Consider a metric measure space $X$. We begin with the definition of
the (global) maximal attractor, which is only slightly different
from Definition~\ref{def:amax}. Let $\G:X\rightarrow X$ be
homeomorphic onto its image, but suppose its image is contained
strictly in $X$. The \emph{(global) maximal attractor} of $\G$ is
defined as:
\begin{equation}    \label{eqn:glob}
A_{max} = \bigcap_{k=0}^\infty \G^k(X)
\end{equation}

Moreover, a measure $\mu_{\infty}$ is called a \emph{good measure}
of $\G$ (with respect to the measure $\mu$ of $X$) if
$$
\mu_\infty = \lim_{k\rightarrow \infty} \frac 1k\sum_{i=0}^{k-1}
\G^i_* \mu,
$$
in the weak topology, see \cite{GI96}.

The closure of the union of supports of all good measures of $\G$ is
called the \emph{minimal attractor}, and it is contained in the
statistical attractor (also see \cite{GI96}). Thus the following
inclusions between attractors hold:

\begin{equation} \label{eqn:incl3}
A_{min} \subset  A_{stat}  \subset  A_{M}  \subset A_{max}.
\end{equation}

An invariant measure $\mu_\infty$ is called an \emph{SRB measure}
with respect to $\mu$ provided that

\begin{equation} \label{eqn:srb}
\int_X \ph d\mu_{\infty} = \lim_{k \rightarrow \infty} \frac
1k\sum_{i=0}^{k-1} \ph(\G^i(x))
\end{equation}
for almost all $x\in X$ and for all continuous functions $\ph\in
C(X)$ (see \cite{P05}). If a good measure is unique and ergodic,
then it is an SRB measure.

The connection between an SRB measure and the $\e-$invisibility
property mentioned in Definition~\ref{def:e-invis} is the following:

\begin{Prop} \label{prop:invissrb}
Consider $X$ and $\G:X\rightarrow X$ as above, and suppose that an
SRB measure $\mu_\infty$ exists. Then an open set $V\subset X$ is
$\e-$invisible if and only if
$$
\mu_\infty(V)\leq \e .
$$
\end{Prop}
\begin{proof} This proposition immediately follows by letting $\ph$
be the characteristic function of $V$ in \eqref{eqn:srb}. Of course,
the characteristic function is not continuous, but it can be
sandwiched between continuous functions arbitrarily tight.
\end{proof}

The classical definitions above traditionally apply to smooth
manifolds $X$, either closed or compact with boundary, for which the
measure $\mu$ is compatible with the smooth structure (a ``Lebesgue
measure''). In the above, we have extended these definitions to
general metric measure spaces. \\

Let $X=\Sigma^2\times S^1$ and $\pi:X\rightarrow S^1$ be the standard projection. Let $C^1_{p,L}$ denote the space of skew products over
the Bernoulli shift with fiber $S^1$, whose fiber maps and their inverses have Lipschitz constant at most $L$. We will
now state the following analogue of the Main Theorem~\ref{thm:mt}
for such skew products, with a quantitative
estimate on the size of the ball $Q_{p,n}$:

\begin{Thm} \label{thm:mtskew} Consider any $n \ge 100$, and let $\nu=\frac 1n$. Then there exists
a ball $Q_{p,n}$ of radius $\frac 1{4n^2}$ in the space $C^1_{p,L}$
with the following property. Any map $G \in Q_{p,n}$ is structurally
stable and has a statistical attractor $A_{stat}=A_{stat}(G)$ such
that the following hold:

\begin{enumerate}

\item the projection $\pi(A_{stat})\subset S^1$ has the property that
\begin{equation}    \label{eqn:proj1}
\pi(A_{stat}) \subset [-\nu, 1 + \nu],
\end{equation}
\item the set $\pi^{-1}(0, \frac 1 4)$ is $\e-$invisible for $G$ with
\begin{equation}    \label{eqn:invis}
 \e = 2^{-n}.
\end{equation}
\end{enumerate}
\end{Thm}

\subsection{North-South skew products}   \label{sub:ns}

The skew products for which we will verify Theorem~\ref{thm:mtskew}
will be from the open set of so-called North-South skew products,
defined below.
\begin{Def}    \label{def:12} A skew product $G:\Sigma^2\times
S^1\rightarrow \Sigma^2\times S^1$ is called a North-South skew product
provided that its fiber maps $g_\om$ have the following properties:

\begin{enumerate}
\item Every map $g_\om$ has one attractor and one repeller, both
hyperbolic.

There exist two non-intersecting closed arcs $I,J\subset S^1$ such
that:

\item All the attractors of the maps $g_\om$ lie strictly inside $I$.

\item All the repellers of the maps $g_\om$ lie strictly inside $J$.

\item All the maps $g_\om$ bring $I$ into itself and are contracting on $I$
uniformly in $\om$. Moreover, the maps $g_\om^{-1} |I$ are
expanding.

\item All the inverse maps $g_\om^{-1}$ bring $J$ into itself and are contracting on $J$
uniformly in $\om$. Moreover, the maps $g_\om |J$ are expanding.

\item The maps $g_\om$ depend continuously on $\om$ in the $C^0$ topology.

\end{enumerate}

\end{Def}

\subsection{Maximal attractors of North-South skew products}   \label{sub:mss}

\begin{Thm}\label{thm:astat} Let $X=\Sigma^2\times S^1$ and let $G:X\rightarrow X$ be a \ns over the Bernoulli shift. Then we have
\begin{description}

\item[a)] The statistical attractor of $G$ is the graph of a continuous
function $\ga = \ga_G: \Sigma^2 \to I.$ The projection
$p|A_{stat}:A_{stat}\rightarrow \Sigma^2$ is a bijection. Under this
bijection, $G|{A_{stat}}$ becomes conjugated to the Bernoulli shift
on $\Sigma^2$:

\begin{equation}   \label{conjugacy A-stat}
\begin{CD}
A_{stat} @>{G}>> A_{stat} \\
@V{p}VV @V{p}VV \\
\Sigma^2 @>{\sigma}>> \Sigma^2
\end{CD}
\end{equation}

\item[b)] There exists an SRB measure $\mu_\infty$ on $X$. This measure is
concentrated on $A_{stat}$ and is precisely the pull-back of the
Bernoulli measure $P $ on $\Sigma^2$ under the bijection
$p|A_{stat}:A_{stat}\rightarrow \Sigma^2$.

\end{description}
\end{Thm}

\begin{proof}
By assumption $4$ of Definition~\ref{def:12}, the map $G$ brings
$\Sigma^2 \times I$ strictly inside itself. We can thus consider the
global maximal attractor of $G|\Sigma^2\times I$:
\begin{equation}   \label{eqn:amax}
A_{max}^* = \bigcap_{k=1}^{\infty} G^k(\Sigma^2\times I).
\end{equation}
We will later prove that
\begin{equation}       \label{eqn:coinc}
A_{max}^* = A_{stat}.
\end{equation}

\begin{Prop}    \label{prop:cont} The attractor $A_{max}^*$ is the graph of a function
$\ga : \Sigma^2 \to I$.
\end{Prop}

\begin{proof}  This follows from assumption $4$ in the
definition of North-South skew products. In more detail, a point
$(\om ,x)$ belongs to $A_{max}^* $ if and only if $(\om ,x)$ belongs
to $G^k(\Sigma^2 \times I)$ for all $k \geq 1 $. This is equivalent
to
\begin{equation} \label{eqn:gamma}
x\in g_{\sigma^{-1}\om }\circ...\circ g_{\sigma^{-k}\om }(I)=:
I_k(\om )
\end{equation}
for all $k\geq 1$. By  assumption $4$, for any fixed $\om $, the
segments $I_k(\om )$ are nested and shrinking as $k \to\infty $.
Hence, in any fiber $\{\om \} \times S^1$, the maximal attractor
$A_{max}^*$ has exactly one point
$$
x(\om ) = \bigcap_{k=1}^{\infty} I_k(\om ).
$$
Define the map $\ga :\Sigma^2 \to I, \ \om \mapsto x(\om )$. By this
definition, $A_{max}^*$ is just the graph of $\ga $.
\end{proof}

\begin{Prop}    \label{prop:cont2} The function $\ga $ defined above
is continuous.
\end{Prop}

\begin{proof}
Consider the notation
$$
g_{k,\om } = g_{\s^{-1}\om } \circ g_{\s^{-2}\om } \circ \dots \circ
g_{\s^{-k}\om }.
$$
By \eqref{eqn:gamma}, we have
$$
\ga (\om ) = \bigcap_{k=0}^\infty g_{k,\om }(I).
$$
Fix a sequence $\om $, fix $\d > 0$ and $m\in \nn$. Let $\om' $ be
so close to $\om $ that :
$$
||g_{\s^{-k}\om } - g_{\s^{-k}\om' }|| \le \delta
$$
for all $k = 1, \dots , m$. Here and in the remainder of this proof,
norms are taken in $C(I)$. Write
$$
\delta_k = ||g_{k,\om } - g_{k,\om' }||.
$$
Then, for $k \le m$,
$$
\delta_k = ||g_{k-1,\om } \circ g_{\s^{-k}\om } - g_{k-1,\om' }
\circ g_{\s^{-k}\om' }|| \le T_1 + T_2,
$$
where
$$
T_1 = ||g_{k-1,\om } \circ g_{\s^{-k}\om } - g_{k-1,\om } \circ
g_{\s^{-k}\om' }||,
$$
$$
T_2 = ||g_{k-1,\om } \circ g_{\s^{-k}\om' } - g_{k-1,\om' } \circ
g_{\s^{-k}\om' }||.
$$
Let $l < 1$ be a common contraction coefficient for all the fiber
maps $g_\om|I$. Then we have
$$
T_1 \le l^{k-1}\delta,
$$
$$
T_2 \le \delta_{k-1}.
$$
The second inequality holds because the fiber maps brings $I$ into
itself and the shift of the argument does not change the $C-$norm.
Therefore, we have
$$
\delta_k\leq \delta_{k-1}+l^{k-1}\delta.
$$
Iterating the above inequality gives us
$$
\delta_m \leq \delta+l\delta+...+l^{m-1}\delta < \frac {\d }{1 - l}.
$$
Therefore, the segments $I_m(\om ), I_m(\om' )$ have length no
greater than $l^m|I|$ and the distance between their corresponding
endpoints is no greater than $\frac {\d}{1-l}$. But this holds for
arbitrarily small $\delta$ and arbitrarily large $m$ when $\om$ and
$\om'$ are close enough. Therefore, \eqref{eqn:gamma} implies that
$\gamma(\om)$ and $\gamma(\om')$ can be made arbitrarily close by
making $\om,\om'$ close enough. This precisely proves the continuity
of $\ga $.
\end{proof}

\subsection{Statistical attractors of \nss}   \label{sub:nss}

Let us now prove \eqref{eqn:coinc}. The proof relies on the
following lemma:
\begin{Lem} \label{lem:Pushing into I}
For almost all $(\om,x)\in \Sigma^2\times S^1$ there exists
$k=k(\om,x) > 0$ such that $G^k(\om,x)\in \Sigma^2\times I$.
\end{Lem}
\begin{proof} On $S^1 \setminus (I \cup J)$ all the fiber maps
$g_\om $ push points away from $J$ and into $I$. Hence, the orbit of
a point $(\om ,x)$ will come to $\Sigma^2 \times I$ if and only if
there exists $k$ \st
$$
G^k(\om ,x) \in \Sigma^2 \times (S^1 \setminus J).
$$
This fails to happen only for elements of the set
$$
S =  \bigcap_{k=0}^\infty G^{-k}(\Sigma^2 \times J).
$$
We will show that the measure of $S$ is zero. Consider the inverse
map
$$
G^{-1}: (\om,x) \mapsto (\sigma^{-1}\om,
g^{-1}_{\sigma^{-1}\om}(x)).
$$
Once again, it is a \ns but the segments $I$ and $J$ now play the
opposite roles: $J$ is contracting and $I$ is expanding. By the
previous section, the maximal attractor $S$ of $G^{-1}|\Sigma^2
\times J$ is the graph of a continuous function
$\gamma^-:\Sigma^2\rightarrow J$. It therefore intersects any fiber
$\{\om\}\times S^1$ at exactly one point. By the Fubini theorem, the
measure of $S$ in $X$ is therefore zero.
\end{proof}

The above lemma shows that the $\om-$limit sets of almost all points
in $X$ belong to $A^*_{max}$. Hence $A_{max}^*$ is the Milnor
attractor of $G$, and thus contains $A_{stat}$. We will now prove
that $A_{max}^*$ is precisely equal to $A_{stat}$.

Consider any good measure $\mu_\infty $ of $G$. For any measurable
set $K\subset \Sigma^2$, we have
$$
G^{-1}(K\times S^1)=\sigma^{-1}(K)\times S^1
$$
and therefore $G_*\mu(K\times S^1)=\mu(\sigma^{-1}(K)\times
S^1)=\mu(K\times S^1)$. Iterating this will give us
$G^k_*\mu(K\times S^1)=\mu(K\times S^1)=P(K)$ for all $k$. By the
definition of good measure this forces
\begin{equation} \label{Formula Mu}
\mu_{\infty}(K\times S^1)=P(K)
\end{equation}
But any good measure is supported on $A_{stat}$, and therefore on
$A_{max}^*$. This and \eqref{Formula Mu} imply that $\mu_{\infty}$
must be the push-forward of $P$ under the isomorphism
$(p|A_{max}^*)^{-1}$. In particular, the support of $\mu_\infty$ is
the whole of $A_{max}^*$.

By the above, the only possible good measure is $\mu_\infty$ given
by \eqref{Formula Mu}. Its support $A_{max}^*$ therefore coincides
with the minimal attractor $A_{min}$. Therefore, by
\eqref{eqn:incl3}, we have that
$$
A_{min}=A_{stat}=A_{max}^*.
$$
This proves statement \textbf{a)} of Theorem \ref{thm:astat}.

Let us now prove statement \textbf{b)} of Theorem \ref{thm:astat}.
We must now show that $\mu_{\infty}=(p|A_{stat})^{-1}_*P$ is an SRB
measure (in particular, our proof will imply that $\mu_{\infty}$ is
a good measure). To this end, we must show that for almost all
$(\om,x)\in X$ and any continuous function $\ph \in C(X)$ we have
\begin{equation} \label{Equation SRB}
\textrm{lim}_{k\rightarrow \infty}\dsp \frac 1k\sum_{i=0}^{k-1}
\ph(G^i(\om,x))= \int \ph d\mu_{\infty}.
\end{equation}
By Lemma~\ref{lem:Pushing into I}, we may restrict attention to
$x\in I$. Then it is easy to note that
$$
\textrm{dist}(G^k(\om,x),G^k(\om,\gamma(\om)))\rightarrow 0
$$
as $k\rightarrow \infty$, uniformly in $\om$ and in $x$. By the
continuity of $\ph$ this implies
$$
\ph(G^k(\om,x))-\ph(G^k(\om,\gamma(\om)))\rightarrow 0
$$
Therefore to prove \eqref{Equation SRB}, it is enough to prove it
for $x=\gamma(\om)$, i.e.
\begin{equation} \label{Equation SRB 2}
\lim_{k\rightarrow \infty}\dsp \frac 1k\sum_{i=0}^{k-1}
\ph(G^i(\om,\gamma(\om)))= \int \ph d\mu_{\infty}
\end{equation}

Since $p:A_{stat}\rightarrow \Sigma^2$ is an isomorphism, the
function $\tilde{\ph}=\ph \circ (p|A_{stat})^{-1}$ is continuous on
$\Sigma^2$. Therefore, \eqref{Equation SRB 2} is equivalent to
$$
\lim_{k\rightarrow \infty}\dsp \frac 1k\sum_{i=0}^{k-1}
\tilde{\ph}(\sigma^i \om)= \int_{\Sigma^2} \tilde{\ph} dP
$$
for almost all $\om$. This statement is just the ergodicity of
$\sigma$, which is a well-known result. We have thus proven that
$\mu_{\infty}$ is an SRB measure, and this concludes the proof of
Theorem~\ref{thm:astat}.
\end{proof}

\subsection{Large $\e-$invisible parts of attractors
for skew products over the Bernoulli shift} \label{sub:large}

In this Section we will complete the proof of
Theorem~\ref{thm:mtskew}. Recall that we
have fixed $n \ge 100$, and let $\nu=\frac 1n$. We shall consider a
particular North-South step skew product $F$, whose fiber maps
$f_0,f_1:S^1\rightarrow S^1$ satisfy the properties listed below:

\begin{enumerate}

\item The maps $f_0,f_1$ each have one attractor, which are $0,1\in \rr/2\zz$
respectively. Suppose further that the arc $I$ in
Definition~\ref{def:12} has the form $I = [-\nu , 1 + \nu ]$ and
that:
\begin{equation}   \label{eqn:contr}
f'_0|I \equiv 1 - \frac {\nu }{2}, \ f'_1|I \equiv \frac 1 4 -  \nu.
\end{equation}

\item The arc $J$ in Definition~\ref{def:12} has the form $J = [-\frac
2 3, -\frac 1 3]$, and the repellers of $f_0, f_1$ are at distance
at least $\nu$ from the endpoints of $J$. We ask that the maps $f_j$
satisfy:
$$
f'_j|J = 1 + \nu .
$$

\item In between the arcs $I$ and $J$, $f_0, f_1$ define a ``one
way'' motion away from $J$ and towards $I$:
$$
f_j(x) \ge x + \frac {\nu^2}{2}  \mbox{, for } x \in [-\frac 1 3,
-\nu ]
$$
$$
f_j(x) \le x -\frac {\nu^2}{2}  \mbox{, for } x\in [1 +\nu, \frac 4
3]
$$

\begin{figure}[ht]
 \begin{center}
  \includegraphics[scale=0.5]{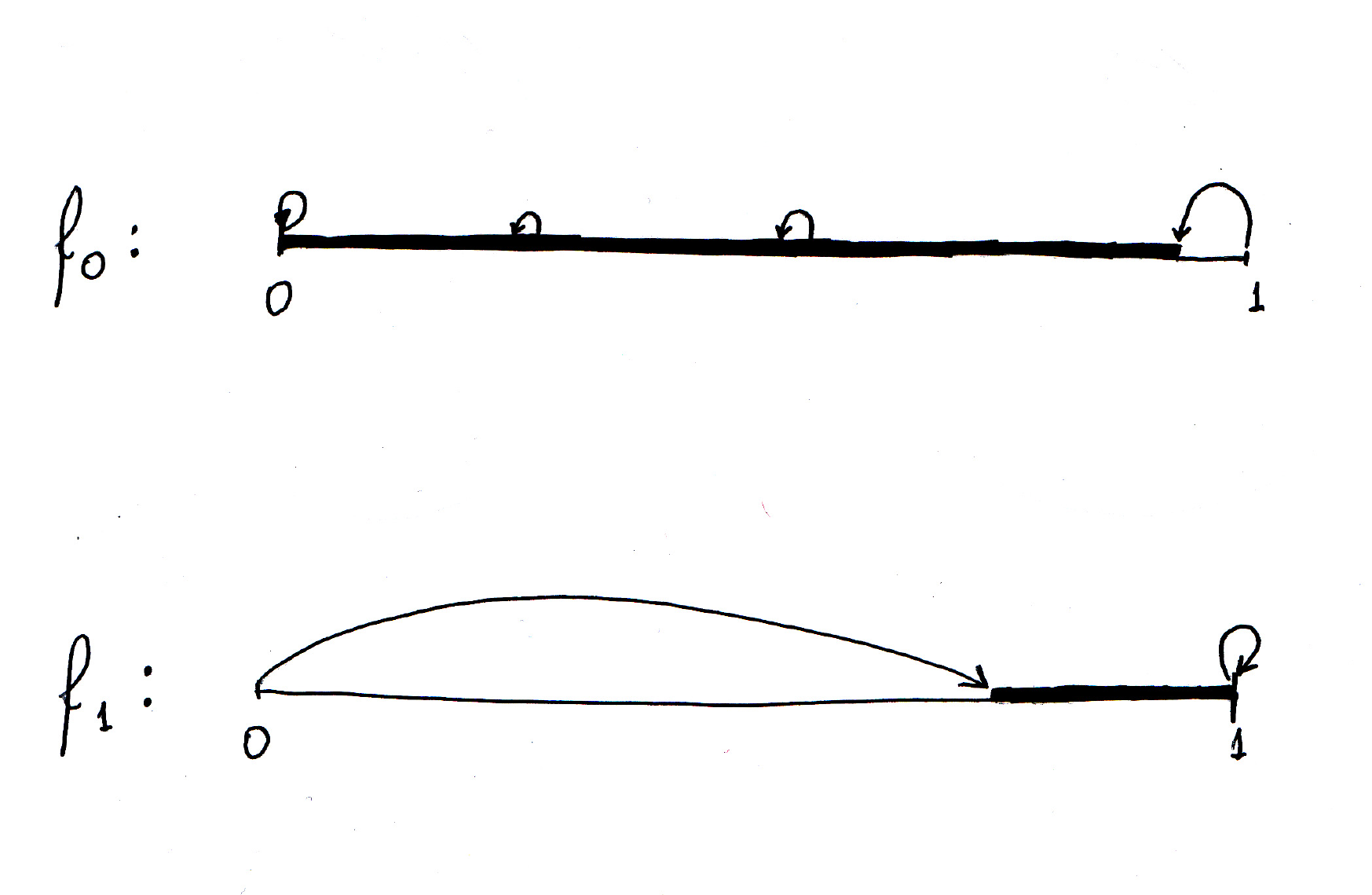}
   \caption{Restrictions of the maps $f_0, f_1$ to the unit segment}
   \label{fig:maps}
  \end{center}
\end{figure}

\end{enumerate}

\begin{Rem} Many of the appearances of $\nu=\frac 1n$ in the above assumptions are due to the fact that we want the qualitative properties of $F$ to survive when we consider perturbations of order $\nu^2$ in the space of skew products (i.e. structural stability). Indeed, it can be shown that this is the case, though we will not get into the technical details.
\end{Rem}

Consider the ball $Q_{p,n}$ of radius $\frac {\nu^2}4$ around $F$ in
the space $C^1_{p,L}$ of skew products over the Bernoulli shift.
This ball consists of skew products $G$ such that
\begin{equation}   \label{eqn:metric}
d(F,G) = \max_{\Sigma^2}d_{C^1}(f^{\pm 1}_\om , g^{\pm 1}_\om ) \leq
\frac {\nu^2}4
\end{equation}
Then any $G \in Q_{p,n}$ has the property that its fiber maps
$g_\om$ satisfy the following:

\begin{enumerate}

\item Any $g_{\om}$ has an attractor $a_\om$ at distance at most $\nu$ from the attractor of $f_{\om_0}$. Moreover, we have
$$
g'_\om|I \in \left[1 - \nu, 1\right) \textrm{, if }\om_0=0,
$$
$$
g'_\om|I \in (0, \frac 14 ] \textrm{, if }\om_0=1.
$$

\item Any $g_{\om}$ has a repeller $a_\om$ at distance at most $\nu$ from the repeller of $f_{\om_0}$. Moreover, we have
$$
g'_\om|J >1 + \frac {\nu}2,
$$

\item Moreover, for any $\om $,
$$
g_\om(x) \ge x+ \frac {\nu^2}{4}  \mbox{ for } x \in [-\frac 1 3,
-\nu ]
$$
\begin{equation} \label{eqn:fone}
g_\om(x) \le x -\frac {\nu^2}{4}  \mbox{ for } x\in [1 +\nu, \frac 4
3]
\end{equation}

\end{enumerate}

All these properties are immediate consequences of the definitions
and of the Implicit
Function Theorem. In particular, it follows that any $G\in Q_{p,n}$ is a North-South skew product. \\

\begin{proof} of Theorem~\ref{thm:mtskew}. The fact that $\pi(A_{stat})\subset I=[-\nu, 1+\nu]$ follows from Theorem~\ref{thm:astat}. Let us now prove that for any $G \in Q_{p,n}$, the set $V =
\pi^{-1}(0, \frac 1 4)$ is $\e-$invisible. We need to check that
almost every point $(\om ,x) \in X$ visits $V$ with frequency no
greater than $\e $. By Lemma~\ref{lem:Pushing into I}, it is enough
to consider $(\om , x) \in \Sigma^2 \times I$.

\begin{Prop}    \label{prop:3} Let $k > n$ and  $(\om , x) \in \Sigma^2 \times
I$ such that $G^k(\om , x) \in V$. Then
$$
(\om_{k-n} \dots \om_{k-1}) = (0 \dots 0).
$$
\end{Prop}

\begin{proof} Let $j\le k - 1$ be minimal \st $\om_{k-j} = 1$. If such a $j$ does not exist or $j > n$,
then the Proposition is proved (since we assumed $k > n$). Suppose
by contraposition that $j \le n$. Then the digit at position zero of
the sequence $\sigma^{k-j}\om $ is $1$. Thus the fiber map
$g_{\sigma^{k-j}\om }$ is $\frac {\nu^2}4-$close to $f_1$, implying:
$$
\pi (G^{k-j+1}(\om , x))= g_{\sigma^{k-j}\om }(\pi
(G^{k-j}(\om,x)))>g_{\sigma^{k-j}\om }(-\nu) >
$$
\begin{equation}      \label{eqn:threeq}
> f_1(-\nu)-\frac {\nu^2}4 =\frac 34+\frac {3\nu}4+\frac {3\nu^2}4 > \frac 3 4.
\end{equation}
By assumption, all the maps $g_{\sigma^l\om }$ for $k - j < l < k$
have digit zero at the zero position, and are thus $\frac
{\nu^2}4-$close to $f_0$. Then \eqref{eqn:threeq} implies that 
$$
\pi (G^k(\om , x))=g_{\sigma^{k-1} \om}\circ... \circ
g_{\sigma^{k-j+1}\om}(\pi (G^{k-j+1}(\om , x))) >
$$
$$
>{(1 -\nu )}^{j-1}\cdot \left(\frac 3 4 - \nu \right) + \nu  =: \varphi (n),
$$
where
 $$
  \varphi (n) = {\left(1 -\frac 1n\right)}^{n-1}\cdot \left(\frac 3 4 - \frac 1 n \right) + \frac 1 n .
$$
We have: $ \varphi (n) = e^{-1}\left(\frac 3 4 - O\left(\frac 1
n\right) \right) > \frac 1 4 $ for large $n$. More accurate
calculation shows that  $ \varphi (n)  > \frac 1 4 $ for $ n \ge
100.$ The inequality $ \varphi (n)  > \frac 1 4 $ contradicts the
assumption of the Proposition.
\end{proof}

The ergodicity of the Bernoulli shift implies that the subword $(0
\dots 0)$ ($n$ zeroes) is met in almost all sequences $\om $ with
frequency $2^{-n}$. This an Proposition~\ref{prop:3} imply that
almost all orbits visit $V$ with frequency at most $\e=2^{-n}$.
Hence $V$ is $\e $-invisible indeed, and this concludes the proof of
Theorem~\ref{thm:mtskew}.

\end{proof}

\section{Skew products over the solenoid}

In this section we construct a map whose smooth perturbations form
the open set $Q_n$ described in Theorem \ref{thm:mt}.

\subsection{The symbolic dynamics and SRB measure for the solenoid map}

Let $h$ be the solenoid map \eqref{eqn:sol}. Denote by $\Lambda$ the
maximal attractor of this map, which is called the
\emph{Smale-Williams solenoid}.  Let $\Sigma^2_1 \subset \Sigma^2$
be the set of infinite sequences of 0's and 1's without a tail of
$1$'s infinitely to the right (i.e. sequences which have 0's
arbitrarily far to the right). Its metric and measure are inherited
from the space $\Sigma^2$. Consider the fate map
\begin{equation} \label{eqn:semi-conj sol}
\Phi:\Lambda\rightarrow \Sigma^2_1, \
\Phi(b)=(...\om_{-1}\om_0\om_1...),
\end{equation}
where we define $\om_k=0$ if $y(h^k(b))\in [0, \frac 1 2)$ and
$\om_k = 1$ if $y(h^k(b)) \in [\frac 1 2, 1)$. The map $\Phi$ is a
bijection with a continuous inverse. Moreover, it conjugates the map
$h|\Lambda$ with the Bernoulli shift:

\begin{equation}   \label{eqn:conj-sol}
\begin{CD}
\Lambda  @>{h}>> \Lambda \\
@V{\Phi}VV @V{\Phi}VV \\
\Sigma^2_1 @>{\sigma}>> \Sigma^2_1
\end{CD}
\end{equation}

In addition to the fate map $\Phi$, we can define the ``forward fate
map'' $\Phi^+(b)=(\om_0\om_1...)$, with $\om_0,\om_1,...$ described
as above. The map $\Phi^+(b)$ is now defined for all $b$ in the
solid torus $B$, and it only depends on $y(b)$. More generally, if
$h^{-k}(b)$ exists, then we can define
$\Phi_{-k}^+(b)=(\om_{-k}...\om_0\om_1...)$.

It is well known that the SRB measure on $\Lambda$ is the pullback
of the Bernoulli measure on $\Sigma_1^2$ under the fate map:

$$\mu_\L = \Phi^* P. $$

\subsection{Attractors  of \nss  over the solenoid}

Let $X=B\times S^1$, where $B$ is the solid torus. A North-South
skew product over the solenoid will refer to a skew product that
satisfies the properties of Definition~\ref{def:12} with
$(\Sigma^2,\om)$ replaced by $(B,b)$.

\begin{Thm}\label{thm:astatsol} Let $\G:X\rightarrow X$ be a North-South skew product over the solenoid. Then

\begin{description}

\item[a)] The statistical attractor of $\G$ lies inside $\Lambda\times I$, and is the graph of a continuous
map $\ga:\Lambda \to I$. Under the projection homeomorphism
$p:A_{stat}\rightarrow \Lambda$, the restriction $\G|{A_{stat}}$
becomes conjugated to the solenoid map on $\Lambda$:

\begin{equation}
\begin{CD}
A_{stat}  @>{\G}>> A_{stat} \\
@V{p}VV @V{p}VV \\
\Lambda @>{h}>> \Lambda
\end{CD}
\end{equation}

\item[b)] There exists an SRB measure $\mu_\infty$ on $X$. This measure is
concentrated on $A_{stat}$ and is precisely the pull-back of the
Bernoulli measure $P$ on $\Sigma^2_1$ under the isomorphism
$\Phi\circ p:A_{stat}\rightarrow \Sigma^2_1$.

\end{description}
\end{Thm}

This theorem is proved in the same way as Theorem~\ref{thm:astat}
with a single difference: we need new arguments to prove the
analogue of Lemma~\ref{lem:Pushing into I}. This will be done in
Lemma~\ref{lem:Pushing into I sol} of the next subsection.

\subsection{Hyperbolicity}

\begin{Lem}    \label{lem:hyp}   Let $\G:X\rightarrow X $ be a \ns over the
solenoid. Then the invariant sets
$$
A=\bigcap_{k=0}^\infty \G^k(B\times I), \textrm{ }\textrm{
}S=\bigcap_{k=0}^\infty \G^{-k}(\Lambda\times J)
$$
are hyperbolic.
\end{Lem}

\begin{Rem} The union $A\cup S$ is the non-wandering set of $\G$.
The set $A$ is a hyperbolic attractor of index 1, while $S$ is a
locally maximal hyperbolic set of index 2. \end{Rem}

This Lemma is a technical result that will be proved shortly. For
now, denote by $W^sS$ the set of all $q \in X$ that attract to $S$
under $\G$:
$$
W^sS =\{ q\in X|d(\G^k(q),S) \to 0 \mbox{ as  } k \to \infty \} .
$$
We claim that set $W^sS$ has measure 0:
\begin{equation} \label{eqn:measzero}
\textrm{mes }W^sS=0.
\end{equation}
This follows from Lemma~\ref{lem:hyp} and Bowen's theorem:
\begin{Thm} [\cite{B75}]  Consider a $C^2$ \diffeo of a compact manifold,
and a hyperbolic invariant set $S$ of this \diffeo which is not a
maximal attractor in its neighborhood. Then the attracting set
$W^sS$ has Lebesgue measure $0$.
\end{Thm}

Now we can prove the following analogue of Lemma~\ref{lem:Pushing
into I}:
\begin{Lem} \label{lem:Pushing into I sol}
For almost all $(b,x)\in B\times S^1$, there exists $k=k(b,x)$ such
that $\G^k(b,x)\in B\times I$.
\end{Lem}

\begin{proof} Note that if the orbit of the point $(b,x)$
eventually escaped $B\times J$, it would be pushed toward
$B\times I$, and finally inside $B\times I$. Therefore the
statement of the Lemma fails only for points whose orbit stays
inside $B\times J$ forever, i.e. for points of the set
$$
T=\bigcap_{k=0}^\infty \G^{-k}(B\times J).
$$
But $T\subset W^sS$, because any point whose orbit stays forever in
$B\times J$ will be attracted to $\Lambda\times J$ (since $B$ is
attracted to $\Lambda$), and thus will be attracted to $S$. This and
\eqref{eqn:measzero} imply that $\textrm{mes }T=0$. This concludes
the proof of the Lemma, and with it the proof of
Theorem~\ref{thm:astatsol}.
\end{proof}

All that remains to prove is Lemma~\ref{lem:hyp}. Let us recall the
definition of hyperbolic sets in the form of the cones condition and
then check it for the invariant sets $A$ and $S$. Here we use
\cite{P04} and \cite{KH95}.

For any $q \in X$ and any subspace $E \subset T_qX$, define the cone
with the axes space $E$ and opening $\a$ to be the set
$$
C(q,E,\a ) =\{ v \in T_qX| \tan \angle (v,E) \leq \a \} .
$$

Suppose that $A$ is an invariant set of a \diffeo $f: X \to X$. We
say that $(A,f)$ satisfy the \textit{cones condition} if the
following holds: there exist two values $\a^\pm $, two
continuous families of cones on $A$:
$$
C^+(q) = C(q, E^+, \a^+), \ C^-(q) = C(q, E^-, \a^-), \ q \in A
$$
and two numbers
$$
0 < \lambda < 1 < \mu
$$
\st for any $q \in A$ the following relations and inequalities hold:
\begin{equation}   \label{eqn:incl}
df_qC^+(q) \subset C^+(f(q)), \ df_q^{-1}C^-(q) \subset
C^-(f^{-1}(q)),
\end{equation}

\begin{equation}   \label{eqn:exp}
|df_qv| \ge \mu |v|, \ v \in C^+(q),
\end{equation}
$$
|df_q^{-1}v| \ge \lambda^{-1}|v|, \ v \in C^-(q).
$$
\begin{Def} A compact invariant set $A$ of a \diffeo $f$ that
satisfies the cones condition above is called \emph{hyperbolic}.
\end{Def}

\begin{proof} of Lemma~\ref{lem:hyp}. Recall that the coordinates on $B$ are $(y,z)$, and the coordinate on the fiber $S^1$ is $x$. The cones condition will be checked in a special metric: we will rescale the coordinates $x$ and $z$ and then use the Euclidian metric in the new
coordinates. This trick works because the Jacobian matrix of the
\skpr over the solenoid is block triangular.

Let $\tilde x = \eta^2x, \ \t z = \eta z$ be new coordinates. Let
$ds^2 = d{\t x}^2 + dy^2 + d{\t z}^2$.  Then, for $\eta>0 $ small,
the matrices $d\G $ and $d\G ^{-1}$ will be almost diagonal:
\begin{equation}   \label{eqn:df}
d\G = \begin{pmatrix} 2 & & \\
& \lambda & \\
& & g'_b
\end{pmatrix} + O(\eta ),
\end{equation}

\begin{equation}   \label{eqn:dfm}
d\G^{-1} = \begin{pmatrix} \frac 1 2 & & \\
& \lambda^{-1} & \\
& & \frac {1}{g'_b \circ g_b^{-1}}
\end{pmatrix} + O(\eta ).
\end{equation}

Conditions \eqref{eqn:incl} and inequalities \eqref{eqn:exp} are
open, so they persist under small perturbations of the operators
$d\G , d\G^{-1}$. Therefore, it is sufficient to check them for the
first diagonal terms in \eqref{eqn:df}, \eqref{eqn:dfm}, and then
they will immediately follow for $d\G,d\G^{-1}$ for $\eta$ small
enough.

\begin{Prop}   \label{prop:parhyp}  Consider the following decomposition of a
vector space: $E = E^+ \oplus E^-$. Let $\A: E \to E$ be a block
diagonal operator corresponding to this decomposition:
$$
\A = \begin{pmatrix}
C & \\
& D \end{pmatrix},
$$
with $||C^{-1}|| \le \mu^{-1}, \ ||D|| \le \lambda , \ 0 < \l <1<
\mu $. Then the cone
$$
C^+ = (0, E^+, \a )=\{(v^+,v^-)\in E \textrm{ such that } |v^-|\leq
\alpha |v^+|\}
$$
for small $\a $ satisfies the following analogs of \eqref{eqn:incl}
and \eqref{eqn:exp}:
\begin{equation}    \label{eqn:conesplus}       
\A C^+ \subset C^+,
\end{equation}

\begin{equation}   \label{eqn:conesplus1}   
|\A v| \ge \frac {\mu }{\sqrt{1+\a^2}}|v| \ \forall v \in C^+.
\end{equation}
For $\a $ small enough, the factor $\frac {\mu }{\sqrt{1+\a^2}}$
will be greater than 1.
\end{Prop}

\begin{proof}  The proof is immediate. Let $v = (v^+, v^-)$ be the
decomposition corresponding to $E = E^+ \oplus E^-$. Then for any $v
\in C^+$,
$$
|{(\A v)}^-| = |Dv^-| \le \l |v^-| \le \l \a |v^+| < \frac {\l }{\mu
}\a |Cv^+| \le \a |Cv^+| = \a |{(\A v)}^+|.
$$
This proves \eqref{eqn:conesplus}. On the other hand, for any $v \in
C^+$,
$$
|Av| = |(Cv^+, Dv^-)| \ge |Cv^+| \ge \mu |v^+| \ge \frac {\mu
}{\sqrt{1+\a^2}}|v|.
$$
This proves \eqref{eqn:conesplus1}.
\end{proof}

We will now prove that the invariant set $A$ of Lemma~\ref{lem:hyp}
satisfies the cones condition. Take any $q = (b,x)\in A$.
Consider
$$
E = T_qX = E^+ \oplus E^-, \ E^+ = \rr \frac {\partial }{\partial
y}, \ E^- = \cc \frac {\partial }{\partial z} \oplus \rr \frac
{\partial }{\partial x}.
$$
Define
$$
C: E^+ \to E^+, \ C:=\textrm{ diag }(2);
$$
$$
D:E^-\rightarrow E^-, \ D=\mbox{ diag }(\l , \l, g'_b(x)).
$$

Since $x\in I$, we have $g'_b(x) < 1$. This splitting and these
operators satisfy the assumptions of Proposition~\ref{prop:parhyp}.
This implies statement of Lemma~\ref{lem:hyp} for $d\G $ and $C^+$
on $A$.

Now let us show that the set $S$ satisfies the cones condition.
Take any $q=(b,x)\in S$, and consider
$$
E = T_qX = E^+ \oplus E^-, \ E^+ = \rr \frac {\partial }{\partial y}
+ \rr \frac {\p }{\p x}, \ E^- = \cc \frac {\p }{\p z}.
$$
Take
$$
C:E^+\rightarrow E^+, \ C= \mbox{ diag }(2, g'_b(x)),
$$
$$
D:E^-\rightarrow E^-, \ D= \mbox { diag }(\l, \l).
$$
As $x \subset J$, we have $g'_b(x) > 1$. Hence, this splitting
and these operators satisfy the assumptions of
Proposition~\ref{prop:parhyp} again. This implies
Lemma~\ref{lem:hyp} for $d\G $ and $C^+$ on $S$.

Similar statements for $d\G^{-1}$ and $C^-$ on $A$ and $S$ are
proved in the exact same way. This concludes the proof of
Lemma~\ref{lem:hyp}, and together with it,
Theorem~\ref{thm:astatsol}.
\end{proof}

\subsection{Almost step skew products over the solenoid}

We now construct an ``almost step'' \skpr over the solenoid, whose
attractor has a large invisible part. Naively, a step skew product
on the solenoid would be a diffeomorphism $\F$ as in
\eqref{eqn:smooth1}, where the fiber maps $f_b$ depend on the digit
$\Phi(b)_0$ only. However, if we set $f_b=f_{\Phi(b)_0}$ for
some fixed diffeomorphisms $f_0,f_1:S^1\rightarrow S^1$, the skew
product would be discontinuous at $y(b)\in \{0,\frac 12\}\subset
S^1$. We must fix this discontinuity.

Consider two diffeomorphisms $f_0,f_1:S^1\rightarrow S^1$, and an
isotopy
$$
f_t:S^1\rightarrow S^1, t\in [0,1]
$$
between them. If $f_0,f_1$ are both orientation preserving, then we
can (and always will) take $f_t = (1 - t^2)f_0 + t^2f_1$.
In this section, numbers in $[0,1)$ are written in binary
representation. For $y\in [0,1)$, define
\begin{equation} \label{eqn:deffibermapssol}
f_y :=\left\{
       \begin{array}{ll}
         f_0, & \textrm{ for } y \in [0, 0.011); \\
         f_{8y-3}, & \textrm{ for } y \in [0.011, 0.1); \\
         f_1, & \textrm{ for } y \in [0.1, 0.111); \\
         f_{8-8y}, & \textrm{ for } y \in [0.111, 1).
       \end{array}
     \right.
\end{equation}
The choice of the isotopy $f_t$ above makes this family $C^1$ in
$y$. The \emph{almost step skew product} over the solenoid,
corresponding to the fiber maps $f_0,f_1$, is defined as
\begin{equation} \label{eqn:stepskewsolen}
\F:X\rightarrow X,  \ \F(b,x)=(h(b), f_{y(b)}(x))
\end{equation}
If $f_0$ and $f_1$ satisfy the properties of
Definition~\ref{def:12}, then $\F$ will be a North South skew
product, see Fig.~\ref{fig:skewsol}. Since we cannot visualize the
4-dimensional phase space, we show on this figure the map
$$
\F' : S^1_{\bold y} \times I \to S^1_{\bold y} \times I, \ (y,x)
\mapsto (y, f_y(x)).
$$

\begin{figure}[ht]
 \begin{center}
  \includegraphics[scale=0.5]{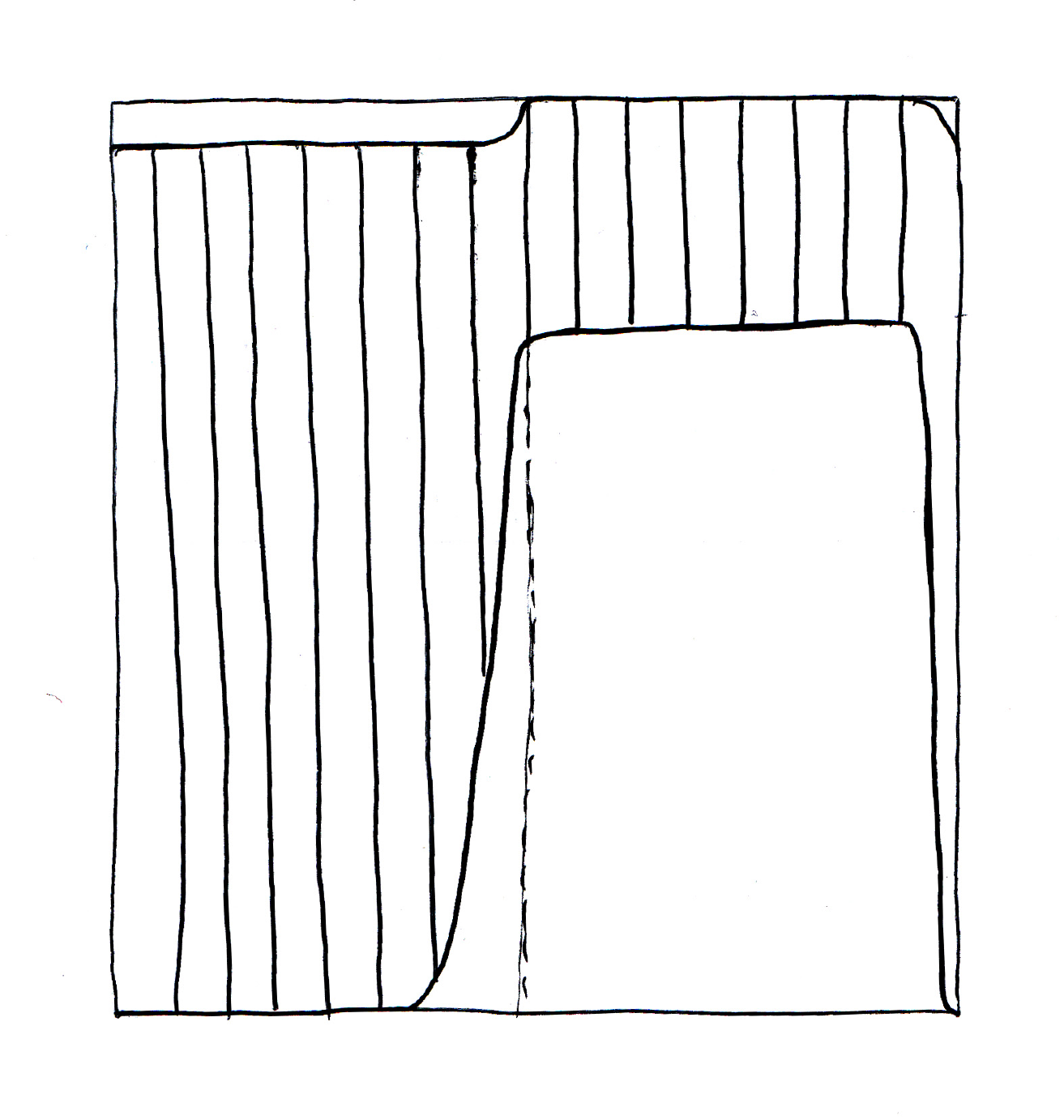}
   \caption{Action of the map $\F' $ and of fiber maps $f_y$ in the family \eqref{eqn:deffibermapssol}.
    Shown are the images of the segment $\{ y\} \times [a_0, a_1]$ under $f_y$}
   \label{fig:skewsol}
  \end{center}
\end{figure}

\begin{Rem}  \label{rem:7} The main feature of almost step \skprs is the following.
Consider a word $w = (\om_0 \dots \om_k)$ that contains no cluster
$11$. Consider a sequence $\om $ with the subword $w$ starting at
the zero position. Let $b = \Phi^{-1}(\om )$. Then
\begin{equation}    \label{eqn:key}
f_{h^{k-1}(b)} \circ \dots \circ f_{b} = f_{\om_{k-1}}\circ
\dots \circ f_{\om_0}.
\end{equation}
Indeed, the binary expansion of $y(h^i(b))$, for any $0 \le i \le k
- 1$, starts with the combination $\om_i\om_{i+1}\om_{i+2}$ which is
different from $011$ or $111$. Hence, by definition, $f_{h^i(b)} =
f_{\om_i}$.
\end{Rem}

Now consider the two diffeomorphisms $f_0,f_1:S^1\rightarrow S^1$ of
Subsection~\ref{sub:large}, and let $\F$ be the almost step skew
product over the solenoid corresponding to these two fiber maps.
Recall that $0$ is the attractor of $f_0$, $1$ is the attractor of
$f_1$, and let $I=[0,1]$. Then the map $f_1\circ f_0$ has a unique
attractor at
$$
a=\frac {\frac 34+\nu }{\frac 34+\nu+\frac {\nu}2\left(\frac 14-\nu
\right)}\in (1-\nu,1)
$$
Let us write $\t I = [0, a]$. The following result establishes the
first statement of Theorem~\ref{thm:mt} for the map $\F$.


\begin{Lem}   \label{lem:10new}
The statistical attractor $A_{stat}$ of $\F$ is a circular arc such
that
$$
\t I \subset \pi(A_{stat}) \subset I.
$$
\end{Lem}

\begin{proof} The attractors of all the fiber maps
\eqref{eqn:deffibermapssol} belong to the segment $I$. Hence the
inclusion on the right follows from Theorem~\ref{thm:astatsol}. We
will now prove the inclusion on the left.

The map $\F $ has a fixed point $q_0$ and a periodic point $q_1$ of
period $2$ described as follows. Let $b_0 \in \Lambda $ be the
unique fixed point of the solenoid map, which has fate
$(...000...)$. Let $b_1 \in \L $ be the periodic point of the
solenoid map with fate $(...010101...)$, with $0$ standing at the
zero position. Then the point $q_0=(b_0,0)$ is fixed by $\F$,
while the point $q_1=(b_1,a)$ has period 2. By Theorem~\ref{thm:astatsol},
$$
A_{stat}=\bigcap_{k=0}^\infty \F^k(B\times I) 
$$ 
From this, it follows that $A_{stat}$ contains all periodic points of $\F$, and
thus $\{0,a\} \subset \pi (A_{stat})$. By
Theorem~\ref{thm:astatsol}, $A_{stat}$ is homeomorphic to the
solenoid, which is a connected set. Therefore, $\pi
(A_{stat})\subset I$ is connected as well, which implies that it is
a circular arc containing $\t I=[0,a]$. \end{proof}

\subsection{Invisible parts of attractors for special \skprs over
the solenoid}     \label{sub:64}

In this section we will complete the proof of Theorem~\ref{thm:mt}
for the map $\F$, by establishing statement 2. Recall that $n\geq
100$ is fixed, and that we denote $\nu=\frac 1n$. \begin{Lem}
\label{lem:invisspec}  The set $V =\pi^{-1}(0,\frac 1 4)$ is $\e
$-invisible for the map $\F$, with $\e = 2^{-n}$. \end{Lem}
\begin{proof} To prove this Lemma, we must show that the orbits of
almost all points $(b,x)\in B\times S^1$ visit $V$ with frequency at
most $\e$. By Lemma~\ref{lem:Pushing into I sol}, we may restrict
attention to $(b,x)\in B\times I$. Let $W$ be the set of finite
words of length $2n$ which do not contain the two-digit sequence
$10$. These words have the form $ 0...01...1.$ The cardinality
of $W$ is clearly $2n+1$.

\begin{Prop}    \label{prop:3sol} Let $k \geq 2n, \ (b , x) \in B
\times I$ and suppose that $\F^k(b , x) \in V$. If $\om=\Phi^+(b)$,
then $$ (\om_{k-2n}...\om_{k-1}) \in W $$ \end{Prop}
\begin{proof}
Suppose by contraposition that the conclusion of the Proposition
fails. Then let $j\leq 2n$ be minimal such that
$\om_{k-j}\om_{k-j+1}=10$. By the definition of $\F$, the fiber map
$f_{h^{k-j}b }$ coincides with $f_1$. This implies
that:
$$
\pi (\F^{k-j+1}(b , x))= f_1(\pi(\F^{k-j}(b,x)))>f_1(0)>\frac 34.
$$
Observe that for any $x\in [0,1]$ and $t\in [0,1]$, we have
$$
f_t(x)=t^2f_1(x)+(1-t^2)f_0(x) \geq f_0(x)=\left(1-\frac
{\nu}2\right) x.
$$
Then we have that
$$
\pi (\F^k(b , x))=f_{h^{k-1} b}\circ... \circ f_{h^{k-j+1}b}(\pi
(\F^{k-j+1}(b , x))) \geq
$$
$$
\geq {\left(1 -\frac {\nu}2\right)}^{j-1}\cdot \frac 3 4 \ge
{\left(1 -\frac 1{2n} \right)}^{2n-1}\cdot \frac 3 4 > \frac 1 4.
$$
The above inequality contradicts the assumption that $\F^k(b,x)\in
V$.
\end{proof}

The ergodicity of the Bernoulli shift implies that subwords in $W$
are met in almost all forward sequences $\om=(\om_0\om_1\om_2...)$
with frequency $2^{-2n}$. But almost all sequences $\om$ correspond
under $\Phi^+$ to almost all $b\in B$. Thus we conclude that, for
almost all $b\in B$, subwords in $W$ are met in $\Phi^+(b)$ with
frequency at most $(2n+1)\cdot 2^{-2n}<2^{-n}=\e$. This and
Proposition~\ref{prop:3sol} imply that almost all orbits visit $V$
with frequency at most $\e$, hence $V$ is $\e-$invisible.
\end{proof}

\section{Perturbations}    \label{sec:pert}

Here we complete the proof of our main result. By Lemmas~\ref{lem:10new} and \ref{lem:invisspec}, we have already proved the conclusion of Theorem~\ref{thm:mt} for the map $\F$ itself. Now we will prove the Theorem for small perturbations of it. \\

\begin{proof} of Theorem~\ref{thm:mt}. We will let $Q_n$ be a small ball in $D_L(X)$ around the almost step skew product $\F$. Thus we have to prove statements 1 and 2 of Theorem~\ref{thm:mt} for any $\G$ which is close enough to $\F$.

Let $I^+=[-\nu, 1+\nu]$. Consider first the maximal attractor of $\G
| B\times I^+$:
$$
A^*_{max }(\G) = \bigcap_{k=0}^\infty \G^k(B \times I^+)
$$
This attractor is connected because $B \times I$ is connected. It
contains all the complete orbits of $\G $, and in particular it
contains fixed points and periodic orbits.

Let $q_0(\F )$ and $q_1(\F )$ be the fixed and periodic points of
$\F$ defined in the proof of Lemma~\ref{lem:10new}. They are
hyperbolic, and thus persist under small perturbations. Hence, the
map $\G $ has a fixed point $q_0(\G )$ and a periodic point $q_1(\G
)$ close to $q_0(\F )$ and $q_1(\F )$, respectively. Moreover, for
$\G $ sufficiently close to $\F $ we will have
$$
\pi (q_0(\G)) \in (-\nu, \nu) , \textrm{ }\textrm{ } \pi (q_1(\G ))
\in (1-\nu, 1+\nu) .
$$
Since $q_0(\G), q_1(\G)\in \pi(A^*_{max}(\G))$ and $A^*_{max}(\G)$
is connected, it follows that $A^*_{max}(\G)$ is a circular arc such
that
\begin{equation} \label{eqn:Amax}
[\nu,1-\nu]\subset \pi(A^*_{max}(\G))\subset [-\nu, 1+\nu].
\end{equation}
By the structural stability of the hyperbolic attractors, $A^*_{max
}(\F)$ is hyperbolic. Since $A^*_{max }(\F) = \Ast (\F)$, the
theorem due to Gorodetski \cite{G96} gives
$$
A^*_{max }(\G) = \Ast (\G).
$$
Hence, \eqref{eqn:Amax} proves conclusion 1 of Theorem~\ref{thm:mt}.

As for conclusion 2, let $\mu_\infty(\F)$ denote the SRB measure for
$\F$ (which is described in Theorem~\ref{thm:astatsol}). By
Lemma~\ref{lem:invisspec} and  Proposition~\ref{prop:invissrb}, it
follows that
$$
\mu_\infty(\F)\left(\pi^{-1}\left(0,\frac 14\right)\right)\leq \e
$$
In fact, by the proof of Lemma~\ref{lem:invisspec} we can even put
$(2n+1)2^{-2n}$ in the right hand side. The Ruelle theorem on the
differentiability of the SRB measure \cite{R97} implies that any
small perturbation $\G$ of $\F$ has an SRB measure $\mu_\infty(\G)$,
and that this measure depends differentiably on $\G$. In particular,
it follows that for $\G$ close enough to $\F$ we will still have
$$
\mu_\infty(\G)\left(\pi^{-1}\left(0,\frac 14\right)\right)\leq \e
$$
By applying Proposition~\ref{prop:invissrb} again, it follows that
$\pi^{-1}(0, \frac 14)$ is $\e-$invisible for $\G$.
\end{proof}

\subsection{Acknowledgments}

The authors are grateful to A. Bufetov, A. Gorodetski, V. Kaloshin,
M. Liubich, J. Milnor, C. Pugh, M. Shub, W. Thurston, B. Weiss for
fruitful comments and discussions. We are also grateful to I. Schurov, who wrote the programs for numeric experiments in which the
existence of $\e-$invisible parts of attractors was first observed. We would like to thank S. Filip and N. Kamburov for their technical help. The second author would like to thank the Math in Moscow program organized by the Independent
University of Moscow for providing a wonderfully stimulating cultural and mathematical experience, which resulted in the writing of this paper.

\end{document}